\documentclass{article}
\usepackage{amsmath,amsthm,amsopn,amstext,amscd,amsfonts,amssymb,verbatim}
\usepackage{color}
\usepackage[numbers]{natbib}

\usepackage{graphicx}
\usepackage{verbatim,color}
\usepackage{animate}

\def\tr{\mathop{\rm tr}\nolimits}

\def\vec{\mathop{\rm vec}\nolimits}

\newcommand {\boldgreektext}[1] {\boldmath
             \(#1\)\unboldmath}
\newcommand {\boldgreek}[1]
             {\mbox{\boldgreektext{#1}}
            }

\renewenvironment{abstract}
                 {\vspace{6pt}
                  \begin{center}
                  \begin{minipage}{5in}
                  \centerline{\textbf{Abstract}}
                  \noindent\ignorespaces
                 }
                 {\end{minipage}\end{center}}
                 
\newtheorem{theorem}{\textbf{Theorem}}[section]
\newtheorem{corollary}{\textbf{Corollary}}[section]

\theoremstyle{definition}

\title{\Large \textbf{Multimatrix variate distributions}}
\author{
  \textbf{Jos\'e A. D\'{\i}az-Garc\'{\i}a} \thanks{Corresponding author\newline
    }\\
  {\normalsize Universidad Aut\'onoma de Chihuahua} \\
  {\normalsize Facultad de Zootecnia y Ecolog\'{\i}a} \\
  {\normalsize Perif\'erico Francisco R. Almada Km 1, Zootecnia} \\
  {\normalsize 33820 Chihuahua, Chihuahua, M\'exico}\\
  {\normalsize E-mail: jadiaz@uach.mx}\\
  \textbf{Francisco J. Caro-Lopera}\\
  {\normalsize University of Medellin} \\
  {\normalsize Faculty of Basic Sciences} \\
  {\normalsize Carrera 87 No.30-65} \\
  {\normalsize Medell\'{\i}n, Colombia} \\
  {\normalsize E-mail: fjcaro@udemedellin.edu.co} \\[2ex]
}
\date{}
\begin{document}
\maketitle

\begin{abstract}
A new family of distributions indexed by the class of matrix variate contoured elliptically distribution is proposed as an extension of some bimatrix variate distributions. The termed \emph{multimatrix variate distributions} open new perspectives for the classical distribution theory, usually based on probabilistic independent models and preferred untested fitting laws. Most of the multimatrix models here derived are invariant under the spherical family, a fact that solves the testing and prior knowledge of the underlying distributions and elucidates the statistical methodology in contrasts with some weakness of current studies as copulas.  The paper also includes a number of diverse special cases, properties and  generalisations. The new joint distributions allows several unthinkable combinations for copulas, such as scalars, vectors and matrices, all of them adjustable to the required models of the experts. The proposed joint distributions are also easily computable, then several applications are plausible. In particular, an exhaustive example in molecular docking on SARS-CoV-2 presents the results on matrix dependent samples.
\end{abstract}

\section{Introduction}\label{sec:1}

The modern world and its paradigms for understanding natural phenomena increasingly demand holistic solutions that involve the consensus of multiple sciences. The so-called new way of doing science, the scientific computing, is based on the postulate that the phenomena of post modernity intertwine multiple dependencies between knowledge and variables. The paradox of the butterfly's flapping effect, for example, confirms that the independence of events in a modelable scientific phenomenon is increasingly difficult to accept in the scientific community. From the epistemological point of view, a construct of independent variables immensely distances reality from the model that seeks to explain it.
Unless an expert manages to completely isolate its universe of explanatory variables of a phenomenon and has proven certainty of its independence, it is generally difficult from the usual theory to arrive at an explanation of the interrelationship in question. Chaos theory and the divergence of non-deterministic phenomena bring modern statistics closer to the elucidation of techniques increasingly in line with the divergent and integrated thinking of the expert. Initially, statistics moved away from interdisciplinary work with science, assuming theories based on probabilistic independence due to the great convenience in the computational solution. Such is the case of the well-known linear models, where the assumption of independent variables allows results as powerful, but difficult to achieve in reality, such as the Gauss-Markov theorem. But perhaps the example of greatest mathematical convenience and widely disseminated in statistics corresponds to the information from probabilistically independent samples that immensely facilitate a definition of tractable likelihood formed by the simple product of the probability densities of each variable. Experts, who deep down know of a high dependence on the samples of their experiments, have had no alternative but to use classical likelihood to approximate the highly dependent phenomena that emerge from their laboratories. A more robust modeling of the real world undoubtedly requires greater planning and statistical analysis of experiments in an interdisciplinary dialogue with the expert, who dictates the main postulates of the phenomena they are trying to decipher. In this sense, statistics does not command decisions and applications, but rather realistically guides the solution that the expert guides. Instead of establishing its classic postulates about independence, normality, estimability, linearity, it seeks to resolve the paradigms that emerge from modern modeling; methodological proposals that increasingly speak of the opposite: dependence, skewed or non-Gaussian distributions, functions of multiple local minima or maxima, nonlinearity and nonparametric statistics, etc. In response to these requirements, statistics has proposed separate theories of great complexity, but they are still not very popular and only exist in the literature that is also frequented by statisticians, not by users of other sciences. The best example speaks of normality, which, like the independence of variables, is almost never met. The normal model has been replaced by elliptical contour distributions, allowing probability distributions to be considered with heavier and lighter tails and a greater or lesser degree of kurtosis, among other significant characteristics. Although a major problem emerges when the underlying distribution is not known a priori, then the decision is left to adjustment criteria, without any explanation from expert science. In this context, invariance under complete families of distributions emerge of utmost importance, so that the results are applied regardless of knowledge of the base distribution.

In some experiments and phenomena where more than one random variable is studied, say $X$ e $Y$, it is assumed that these variables are probabilistically independent with densities $dF_{X}(x)$ and $dG_{Y}(y)$. Then, if we are interested in the joint density of $X$ and $Y$, this reduces to the elementary product $dF_{X}(x)dG_{Y}(y)$. However, as we have said, modern chaos theories and divergent modeling of complex holistic problems warn against this assumption of independence, in which case finding the joint density is not an easy task. For example, in a study of climatological data it is of interest to study the random variables associated with temperature and
precipitation. For practical purposes these variables are considered independent, with normal and gamma distributions, respectively. But it is known that the temperature and precipitation are not probabilistically independent random variables. Then we open the question about a joint distribution of $X$ and $Y$ when the variables are not probabilistically independent. Furthermore, it would be desirable that given
the joint density function of the probabilistically dependent variables $X$ and $Y$, the corresponding marginal distributions of the random variables $X$ and $Y$ were the marginal distributions that are usually assumed under independence.For the univariate case, several bivariate distributions have been studied by different authors, the models include gamma-beta, gamma-gamma, beta-beta among others. They have found applications in problems of hydrology, finance and other areas disciplines, see \citet{ln:82},\citet{cn:84}, \citet{n:07, n:13}, \citet{ol:03} and \citet{spj:14} and references therein. In general, in these works, joint distributions of dependent random variables are constructed, making changes of variables on independent random variables. From the matrix point of view bimatrix variate type distribution have been studied by several authors, see \citet{or:64}, \citet{dggj:10a, dggj:10b,dggj:11}, and \citet{brea:11} among others. \citet{e:11} presents a very complete exposition of the advances in matrix variate case up to that moment. The authors have studied this problem for scalar and vector cases in \citet{dgclpr:22} and a matrix case in the multimatricvariate version in \citet{dgcl:22}.
Finally, a multimatrix theory that solves the previous problem but that cannot be computed would add to the hundreds of theoretical results in matrix analysis developed in the last 70 years. Then we expect that the joint distribution functions of several matrices are free of approximations and truncation of complex series of polynomials of matrix arguments. 

The above discussion is placed in the present article by proposing a new method to establish computable joint density functions of two or more random matrices, which shall be termed \textit{multimatrix variate distributions}\footnote{\citet{gm:93}proposed the term \textit{matrix multivariate} instead of \textit{matrix variate}, commonly used in multivariate analysis literature, \citet{gv:93}, among many other authors. The term \textit{matricvariate} was introducing by \citet{di:67} to distinguish between the two possible expressions of the density  matrix multivariate $T$-distribution. A detail explanation about this topic is given in \citet[see Section 3]{dggj:06}.}, such that the corresponding marginal distributions are known. Section \ref{sec:1} establishes the notation to be used and some preliminary results on the calculation of Jacobians and some integrals of interest as well as the definition of the matrix variate elliptical contoured distributions. The main distributional results are collected in Section \ref{sec:2}. Some related distributions  of multimatrix variate distributions and particular generalisations  are obtained in Section \ref{sec:3}. Finally, Section \ref{sec:4} provides a dependent sample emergent from a problem of molecular docking into a new cavity of SARS-CoV-2. In particular, the joint distribution of the dependent sample is estimated by a simple optimisation routine.

\section{Preliminary results}\label{sec:2}

Matrix notations, matrix variate elliptical contoured distributions, and Jacobians are presented in this section.
First, we start with notations and terminologies. $\mathbf{A}\in \Re^{n \times m}$ denotes a \emph{matrix} with $n$
rows and $m$ columns; $\mathbf{A}'\in \Re^{m \times n}$ is the \emph{transpose matrix}, and if $\mathbf{A}\in \Re^{n \times n}$ has an \emph{inverse}, it shall be denoted by $\mathbf{A}^{-1} \in \Re^{n \times n}$. $\mathbf{A}\in \Re^{n \times n}$ is a \emph{symmetric matrix} if $\mathbf{A} = \mathbf{A}'$. If all their eigenvalues are positive then $\mathbf{A}$ is a \emph{positive definite
matrix},  a fact denoted as $\mathbf{A} > \mathbf{0}$. An \emph{identity matrix} shall be denoted by $\mathbf{I}\in \Re^{n \times n}$. To specify the size of the identity, we shall use $\mathbf{I}_{n}$. $\tr (\mathbf{A})$ denotes the trace of matrix $\mathbf{A} \in \Re^{m \times m}$. If $\mathbf{A}\in \Re^{n \times m}$ then by $\vec (\mathbf{A})$ we mean the $mn \times 1$ vector formed by stacking the columns of $\mathbf{A}$ under each other; that is, if $\mathbf{A} = [\mathbf{a}_{1}\mathbf{a}_{2}\dots
\mathbf{a}_{m}]$, where $\mathbf{a}_{j} \in \Re^{n \times 1}$ for $j = 1, 2, \dots,m$
$$
  \vec(\mathbf{A})= \left [
                     \begin{array}{c}
                       \mathbf{a}_{1} \\
                       \mathbf{a}_{2} \\
                       \vdots \\
                       \mathbf{a}_{m}
                     \end{array}
                    \right ].
$$
The \textit{Frobenius norm} of a matrix $\mathbf{A}$ shall be denoted as $||\mathbf{A}||$. Typically the
Frobenius norm is denoted as $||\mathbf{A}||_{F}$, to differentiate it from other matrix norms. Since
we shall use only the Frobenius norm, it just be denoted as $||\mathbf{A}||$. It is
defined by
$$
  ||\mathbf{A}|| = \sqrt{\tr(\mathbf{A}'\mathbf{A})}=\sqrt{\vec'(\mathbf{A})\vec(\mathbf{A})}.
$$

Finally, $\mathcal{V}_{m,n}$ denotes the \emph{Stiefel manifold}, the space of all matrices
$\mathbf{H}_{1} \in \Re^{n \times m}$ ($n \geq m$) with orthogonal columns, that is, $\mathcal{V}_{m,n} =
\{\mathbf{H}_{1} \in \Re^{n \times m}; \mathbf{H}'_{1}\mathbf{H}_{1} = \mathbf{I}_{m}\}$. In addition, if
$(\mathbf{H}'_{1}d\mathbf{H}_{1})$ defines an \textit{invariant measure on the Stiefel manifold}
$\mathcal{V}_{m,n}$, from Theorem 2.1.15, p. 70 in \citet{mh:05},
\begin{equation}\label{evs}
  \int_{\mathcal{V}_{m,n}} (\mathbf{H}'_{1}d\mathbf{H}_{1}) = \frac{2^{m} \pi^{mn/2}}{\Gamma_{m}[n/2]}.
\end{equation}
where $\Gamma_{m}[a]$ denotes the multivariate Gamma function, see \citet[Definiton 2.1.10, p. 61]{mh:05}

Next a result about Jacobians is derived.

\begin{theorem}\label{teoJ}
Suppose that $\mathbf{X} \in \Re^{n \times m}$ and $\mathbf{Y} \in \Re^{n \times m}$ are matrices with
mathematically independent elements.
\begin{description}
  \item[i)] Let $\mathbf{Y} = (1- \tr \mathbf{X}'\mathbf{X})^{-1/2}\mathbf{X}$. Then
     \begin{equation}\label{eq2}
        (d\mathbf{Y}) = (1- \tr \mathbf{X}'\mathbf{X})^{-(nm/2+1)} (d\mathbf{X}).
     \end{equation}
  \item[ii)] If $\mathbf{X} = (1+ \tr\mathbf{Y}'\mathbf{Y})^{-1/2}\mathbf{Y}$, we have
     \begin{equation}\label{eq3}
        (d\mathbf{X}) = (1+ \tr\mathbf{Y}'\mathbf{Y})^{-(nm/2+1)} (d\mathbf{Y}).
     \end{equation}
\end{description}
\end{theorem}
\begin{proof}
\textbf{i}) The proof follows from Theorem 1, \citet{dgclpr:22} observing that $\mathbf{Y} = (1- \tr
\mathbf{X}'\mathbf{X})^{-1/2}\mathbf{X}$ can be written as $\mathbf{y} =
(1-\mathbf{x}'\mathbf{x})^{-1/2}\mathbf{x}$, where $\mathbf{x} = \vec \mathbf{X} \in \Re^{mn}$ and
$\mathbf{y} = \vec \mathbf{Y} \in \Re^{mn}$. \textbf{ii}). Its proof is analogous to one given for \textbf{i}).
\end{proof}

Now, let $\mathbf{V} \in \Re^{N \times m}$ random matrix with a \textit{matrix variate elliptical
distribution} with respect to the Lebesgue measure $(d\mathbf{V})$. Therefore its density function is
given by
\begin{equation}\label{elliptical}
 dF_{_{\mathbf{V}}} (\mathbf{V}) = |\mathbf{\Sigma}|^{-N/2}|\mathbf{\Theta}|^{-m/2}
  h\left\{\tr\left[(\mathbf{V}-\boldsymbol{\mu})^{T}\mathbf{\Sigma}^{-1}(\mathbf{V}-
  \boldsymbol{\mu})\mathbf{\Theta}^{-1}\right]\right\}(d\mathbf{V}).
\end{equation}
The location  parameter is $\boldsymbol{\mu}\in \Re^{N \times m}$; and the scale parameters
$\mathbf{\Sigma}\in \Re^{N \times N}$ and $\mathbf{\Theta}\in \Re^{m \times m}$, are positive definite
matrices. The distribution shall be denoted by $\mathbf{V}\sim \mathcal{E}_{N \times m}(\boldsymbol{\mu}
,\mathbf{\Sigma}, \mathbf{\Theta}; h)$, and indexed by the kernel function $h\mbox{: } \Re \to
[0,\infty)$, where $\int_{0}^\infty u^{Nm/2-1}h(u)du < \infty$.

When $\boldsymbol{\mu} = \mathbf{0}$, $\mathbf{\Sigma} = \mathbf{I}_{N}$ and $\mathbf{\Theta} =
\mathbf{I}_{m}$ a special case of a matrix variate elliptical distribution appears, in this case it is
said that $\mathbf{V}$ has a \textit{matrix variate spherical distribution}.

Note that for constant $a \in \Re$, the substitution $v = u/a$ and $(du) = a(dv)$ in \citet[Equation 2.21, p. 26]{fzn:90} provides that
\begin{equation}\label{int}
    \int_{v > 0} v^{Nm/2-1}h(a v) (dv)= \frac{a^{-Nm/2} \Gamma_{1}[Nm/2]}{\pi^{Nm/2}}.
\end{equation}

\section{Multimatrixvariate distributions}\label{sec:3}

Assume that $\mathbf{X}\sim \mathcal{E}_{N \times m}(\mathbf{0}, \mathbf{I}_{N}, \mathbf{I}_{m}; h)$,
such that $n_{0}+n_{1}+ \cdots + n_{k} = N$, and $\mathbf{X} = \left(\mathbf{X}'_{0},\mathbf{X}'_{1},
\dots, \mathbf{X}'_{k} \right)'$, $n_{i} \geq m \geq 1$. Then (\ref{elliptical}) can be written as
\begin{equation}\label{eq0}
    dF_{\mathbf{X}_{0},\mathbf{X}_{1}, \dots,\mathbf{X}_{k}}(\mathbf{X}_{1}, \dots,\mathbf{X}_{k}) =
     h[\tr(\mathbf{X}'_{0}\mathbf{X}_{0}+ \mathbf{X}'_{1}\mathbf{X}_{1}+\cdots+\mathbf{X}'_{k}
     \mathbf{X}_{k})]\bigwedge_{i=0}^{k}(d\mathbf{X}_{i}),
\end{equation}
or in terms of the Frobenius norm, we can rewrite (\ref{eq0}) as
\begin{equation*}
    dF_{\mathbf{X}_{0},\mathbf{X}_{1}, \dots,\mathbf{X}_{k}}(\mathbf{X}_{1}, \dots,\mathbf{X}_{k}) =
     h\left(||\mathbf{X}_{0}||^{2}+ ||\mathbf{X}_{1}||^{2}+\cdots+||\mathbf{X}_{k}||^{2}\right )
     \bigwedge_{i=0}^{k}(d\mathbf{X}_{i}),
\end{equation*}
where $\mathbf{X}_{i} \in \Re^{n_{i} \times m}$, $i = 0,1, \dots, k$.
Note that the random matrices $\mathbf{X}_{0}, \mathbf{X}_{1}, \dots,\mathbf{X}_{k}$ are
probabilistically dependent. Furthermore, only under a matrix variate normal distribution, the random matrices are independent, see \citet{fz:90}. 

Analogously to \citet[Equation 4.2]{dgclpr:22} but now defining $V_{i} = ||\mathbf{X}_{i}||^{2}$, $i = 0,1,\dots,k$, it is
obtained
\begin{equation}\label{gg}
    dF_{V_{0},\ldots,V_{k}}(v_{0},\ldots,v_{k})=\frac{\pi^{Nm/2}}
    {\displaystyle\prod_{i=0}^{k}\Gamma\left[n_{i}m/2\right]}
    h\left(\sum_{i=0}^{k}v_{i}\right)\prod_{i=0}^{k}v_{i}^{n_{i}m/2-1}\bigwedge_{i=0}^{k}(du_{i}),
\end{equation}
where $N = n_{0}+\cdots+n_{k}$, this distribution shall be termed \emph{multivariate generalised Gamma distribution} and its univariate case ($i=0$)  shall be termed \emph{generalised Gamma distribution}.

\begin{theorem}\label{gge} Suppose that $\mathbf{X} = \left(\mathbf{X}'_{0}, \dots,
\mathbf{X}'_{k} \right)'$ has  a matrix variate spherical distribution, with $\mathbf{X}_{i} \in
\Re^{n_{i} \times m}$, $n_{i} \geq m$, $i = 0,1, \dots, k$. Define $V = ||\mathbf{X}_{0}||^{2}$. Then,
the joint density $dF_{V,\mathbf{X}_{1}, \dots,\mathbf{X}_{k}}(v,\mathbf{X}_{1}, \dots,\mathbf{X}_{k})$
is given by
\begin{equation}\label{mgge}
   \frac{\pi^{n_{0}m/2}}{\Gamma_{1}[n_{0}m/2]} h\left[v+\displaystyle\sum_{i=1}^{k}||\mathbf{X}_{i}||^{2}\right]
    v^{n_{0}m/2-1} (dv)\bigwedge_{i=1}^{k}\left(d\mathbf{X}_{i}\right),
\end{equation}
where $V > 0$. This distribution shall be termed \emph{multimatrix variate generalised Gamma - Elliptical
distribution}.
\end{theorem}
\begin{proof}
We have that 
\begin{equation*}
    dF_{\mathbf{X}_{0},\mathbf{X}_{1}, \dots,\mathbf{X}_{k}}(\mathbf{X}_{1}, \dots,\mathbf{X}_{k}) =
     h\left(||\mathbf{X}_{0}||^{2}+ ||\mathbf{X}_{1}||^{2}+\cdots+||\mathbf{X}_{k}||^{2}\right )
     \bigwedge_{i=0}^{k}(d\mathbf{X}_{i}).
\end{equation*}
Define $V = ||\mathbf{X}_{0}||^{2}$ and from \citet[Theorem 2.1.14, p. 65]{mh:05}, we have that $(d\mathbf{X}_{0}) = 2^{-1}v^{n_{0}m/2 - 1}(dv)\wedge(\mathbf{h}_{1}d\mathbf{h}_{1})$, where $\mathbf{h}_{1} \in \mathcal{V}_{1,n_{0}m}$. \newline Thus, the density  $dF_{V, \mathbf{h}_{1},\mathbf{X}_{1}, \dots,\mathbf{X}_{k}}(v, \mathbf{h}_{1},\mathbf{X}_{1}, \dots,\mathbf{X}_{k})$ is
\begin{equation*}
       \frac{v^{n_{0}m/2 - 1}}{2} h\left(v+ ||\mathbf{X}_{1}||^{2}+\cdots+||\mathbf{X}_{k}||^{2}\right )
      (dv)\wedge (\mathbf{h}_{1}d\mathbf{h}_{1})\bigwedge_{i=1}^{k}(d\mathbf{X}_{i}).
\end{equation*}
Integration over $\mathbf{h}_{1} \in \mathcal{V}_{1,n_{0}m}$ by using (\ref{evs}) gives the required result. 
\end{proof}

\begin{theorem}\label{ggt}
Assume that $\mathbf{X} = \left(\mathbf{X}'_{0}, \dots,\mathbf{X}'_{k} \right)'$ has a matrix variate
spherical distribution, with $\mathbf{X}_{i} \in \Re^{n_{i} \times m}$, $n_{i} \geq m$, $i = 0,1, \dots,
k$. Define $V = ||\mathbf{X}_{0}||^{2}$ and $\mathbf{T}_{i} = V ^{-1/2}\mathbf{X}_{i}$, $i = 1,\dots,k$.
The joint density $dF_{V,\mathbf{T}_{1}, \dots,\mathbf{T}_{k}}(v,\mathbf{T}_{1}, \dots,\mathbf{T}_{k})$
is given by
\begin{equation}\label{mggp7}
   \frac{\pi^{n_{0}m/2}}{\Gamma_{1}[n_{0}m/2]} h\left[v\left(1+\displaystyle\sum_{i=1}^{k}||\mathbf{T}_{i}||^{2}\right)\right]
    v^{Nm/2-1} (dv)\bigwedge_{i=1}^{k}\left(d\mathbf{T}_{i}\right),
\end{equation}
where $N = n_{0}+n_{1}+\cdots+n_{k}$, $V > 0$ and $\mathbf{T}_{i} \in \Re^{n_{i} \times m}$, $i = 1,\dots,k$. This distribution shall be termed \emph{multimatrix variate generalised Gamma - Pearson type VII distribution}. Moreover, the termed  \emph{multimatrix variate Pearson type VII} is the marginal density $dF_{\mathbf{T}_{1}, \dots,\mathbf{T}_{k}}(\mathbf{T}_{1}, \dots,\mathbf{T}_{k})$ of $\mathbf{T}_{1},
\dots,\mathbf{T}_{k}$ and is given by
\begin{equation}\label{mp7}
   \frac{\Gamma_{1}[Nm/2]}{\pi^{(N-n_{0})m/2}\Gamma_{1}[n_{0}m/2]}
  \left(1+\displaystyle\sum_{i=1}^{k}||\mathbf{T}_{i}||^{2}\right)^{-Nm/2}
  \bigwedge_{i=1}^{k}\left(d\mathbf{T}_{i}\right),
\end{equation}
where $\mathbf{T}_{i} \in \Re^{n_{i} \times m}$, $n_{i} \geq m$ and $N = n_{0}+n_{1}+\cdots+n_{k}$.
\end{theorem}
\begin{proof}
The density (\ref{mggp7}) follows from (\ref{mgge}) by defining $\mathbf{T}_{i}= V^{-1/2}\mathbf{X}_{i}$, $i = 1,\dots,k$. Hence by \citet[Theorem 2.1.4, p. 57]{mh:05}, we have
$$
   \bigwedge_{i=1}^{k}(d\mathbf{X}_{i})= \prod_{i=1}^{k}v^{n_{i}m/2}\bigwedge_{i=1}^{k}(d\mathbf{T}_{i}).
$$
And then the required result follows.
Now, integrating (\ref{mggp7}) over $V>0$, by using (\ref{int}), provides the density (\ref{mp7}).
\end{proof}

\begin{theorem}\label{ggpII}
Suppose that $\mathbf{X} = \left(\mathbf{X}'_{0}, \dots, \mathbf{X}'_{k} \right)'$ has a matrix variate
spherical distribution, with $\mathbf{X}_{i} \in \Re^{n_{i} \times m}$, $n_{i} \geq m$, $i = 0,1, \dots,
k$. Define $V = ||\mathbf{X}_{0}||^{2}$ and $\mathbf{R}_{i} = \left(V +
||\mathbf{X}_{i}||^{2}\right)^{-1/2} \mathbf{X}_{i}$, $i = 1,\dots,k$. Then the joint density of
$V,\mathbf{R}_{1}, \dots,\mathbf{R}_{k}$, denoted as $dF_{V,\mathbf{R}_{1},
\dots,\mathbf{R}_{k}}(v,\mathbf{R}_{1}, \dots,\mathbf{R}_{k})$, is given by
$$
  \frac{\pi^{n_{0}m/2}}{\Gamma_{1}[n_{0}m/2]} v^{Nm/2-1} h\left[ v\left( 1+\displaystyle\sum_{i=1}^{k}
  \frac{||\mathbf{R}_{i}||^{2}}{\left(1- ||\mathbf{R}_{i}||^{2}\right)}\right)\right ]\hspace{4cm}
$$
\begin{equation}\label{ggp2}
  \hspace{4cm}
  \times  \prod_{i=i}^{k}\left(1-
  ||\mathbf{R}_{i}||^{2}\right)^{-n_{i}m/2-1}
  (dv)\bigwedge_{i=1}^{k}\left(d\mathbf{R}_{i}\right),
\end{equation}
where $N = n_{0}+n_{1}+\cdots+n_{k}$, $V > 0$ and $\mathbf{R}_{i} \in \Re^{n_{i} \times m}$, and $||\mathbf{R}_{i}||^{2} \leq 1$ $i =
1,\dots,k$. This distribution shall be termed \emph{multimatrix variate Generalised Gamma-Pearson type II
distribution}. Also,  the marginal density $dF_{\mathbf{R}_{1}, \dots,\mathbf{R}_{k}}(\mathbf{R}_{1},
\dots,\mathbf{R}_{k})$ is
$$
  \frac{\Gamma_{1}[Nm/2]}{\pi^{(N-n_{0})m/2}\Gamma_{1}[n_{0}m/2]} \left[ 1+\displaystyle\sum_{i=1}^{k}
  \frac{||\mathbf{R}_{i}||^{2}}{\left(1- ||\mathbf{R}_{i}||^{2}\right)}\right ]^{-Nm/2} \hspace{4cm}
$$
\begin{equation}\label{mp2}
  \hspace{4cm}\times  \prod_{i=i}^{k}\left(1-
  ||\mathbf{R}_{i}||^{2}\right)^{-n_{i}m/2-1}
  \bigwedge_{i=1}^{k}\left(d\mathbf{R}_{i}\right),
\end{equation}
which shall be termed \emph{multimatrix variate Pearson type II distribution}.
\end{theorem}
\begin{proof}
Consider the substitution $\mathbf{T}_{i}=(1-||\mathbf{R}_{i}||^{2})^{-1/2}\mathbf{R}_{i}$, $i=1,\dots,k$ in (\ref{mggp7}) , hence by Theorem \ref{teoJ},
$$
  \bigwedge_{i=1}^{k}\left(d\mathbf{T}_{i}\right) = \prod_{i=i}^{k}\left(1- ||\mathbf{R}_{i}||^{2}\right)^{-n_{i}m/2-1} \bigwedge_{i=1}^{k}\left(d\mathbf{R}_{i}\right),
$$
and the desired result follows. Now (\ref{mp2}) follows by integration of (\ref{ggp2}) over $V>0$ via (\ref{int}).
\end{proof}

\begin{theorem}\label{mggb2}
Consider the hypotheses of Theorem \ref{ggt} and define $\mathbf{F}_{i} = \mathbf{T}'_{i}\mathbf{T}_{i} > \mathbf{0}$, $i = 1, \dots,k$. Then the joint density $dF_{V, \mathbf{F}_{1}, \dots,\mathbf{F}_{k}}(v,\mathbf{F}_{1}, \dots,\mathbf{F}_{k})$ is given by
\begin{equation}\label{ggb2}
   \frac{\pi^{Nm/2}v^{Nm/2-1}}{\Gamma_{1}[n_{0}m/2]}   \prod_{i=1}^{k} \left(\frac{|\mathbf{F}_{i}|^{(n_{i}-m-1)/2}}{\Gamma_{m}[n_{i}/2]}\right )
   h\left[v\left(1+\displaystyle\sum_{i=1}^{k}\tr\mathbf{F}_{i}\right)\right]
   (dv)\bigwedge_{i=1}^{k}\left(d\mathbf{F}_{i}\right).
\end{equation}
This distribution shall be termed \emph{multimatrix variate Generalised Gamma-beta type II distribution}.
\newline The corresponding marginal density function
 $dF_{\mathbf{F}_{1}, \dots, \mathbf{F}_{k}}(\mathbf{F}_{1}, \dots,\mathbf{F}_{k})$
is given by
\begin{equation}\label{mb2}
  \frac{\Gamma_{1}[Nm/2]}{\Gamma_{1}[n_{0}m/2]\displaystyle\prod_{i=1}^{k}\Gamma_{m}[n_{i}/2]}
  \prod_{i=1}^{k}|\mathbf{F}_{i}|^{(n_{i}-m-1)/2}
  \left(1+\displaystyle\sum_{i=1}^{k}\tr\mathbf{F}_{i}\right)^{-Nm/2}
  \bigwedge_{i=1}^{k}\left(d\mathbf{F}_{i}\right).
\end{equation}
and shall be termed \emph{multimatrix variate beta type II distribution}.
\end{theorem}
\begin{proof}
Joint density functions (\ref{ggb2}) and (\ref{mb2}) follow from substitutions $\mathbf{F}_{i} = \mathbf{T}'_{i}\mathbf{T}_{i}$, $i =1,\dots,k$ in  (\ref{mggp7}) and (\ref{mp7}), respectively. First, note that \citet[Theorem 2.1.14, p. 66]{mh:05} provides
$$
  \bigwedge_{i=1}^{k}\left(d\mathbf{T}_{i}\right) = 2^{-mk}  \prod_{i=1}^{k}|\mathbf{F}_{i}|^{(n_{i}-m-1)/2} \bigwedge_{i=1}^{k}\left(d\mathbf{F}_{i}\right)\bigwedge_{i=1}^{k}\left(\mathbf{H}'_{1_{i}}d\mathbf{H}_{1_{i}}\right)
$$ 
where $\mathbf{H}_{1_{i}}\in \mathcal{V}_{n_{i},m}$, $i=1,\dots,k$. Hence, integrating over $\mathbf{H}_{1_{i}}\in \mathcal{V}_{n_{i},m}$ $i=1,\dots,k$ by (\ref{evs}) turns 
$$
  \int_{\mathbf{H}_{1_{1}}} \cdots \int_{\mathbf{H}_{1_{k}}} \bigwedge_{i=1}^{k}\left(\mathbf{H}'_{1_{i}}d\mathbf{H}_{1_{i}}\right) = \frac{2^{mk} \pi^{(N-n_{0})m/2}}{\displaystyle\prod_{i=1}^{k} \Gamma_{m}[n_{i}/2]}
$$
and the required result is derived. 
\end{proof}

A particular case of the density function (\ref{mb2}) is proposed in \citet[Problem 3.18, p.118]{mh:05}.

\begin{theorem}\label{ggbI}
Assuming that $\mathbf{B}_{i} = \mathbf{R}'_{i}\mathbf{R}_{i} > \mathbf{0}$ and $\tr \mathbf{B}_{i} \leq 1$ with $i = 1, \dots,k$, in Theorem
\ref{ggpII}. Then the joint density $dF_{V, \mathbf{B}_{1},\dots,\mathbf{B}_{k}}(v, \mathbf{B}_{1},
\dots,\mathbf{B}_{k})$ is
$$
  \frac{\pi^{Nm/2} \ v^{Nm/2-1}}{\Gamma_{1}[n_{0}m/2]}
  \prod_{i=1}^{k}\left(\frac{|\mathbf{B}_{i}|^{(ni-m-1)/2}}{\Gamma_{m}[n_{i}/2]}\right )
  h\left[ v\left( 1+\displaystyle\sum_{i=1}^{k} \frac{\tr\mathbf{B}_{i}}{\left(1-
  \tr\mathbf{B}_{i}\right)}\right)\right ]\hspace{3cm}
$$
\begin{equation}\label{ggb1}
  \hspace{5cm}
  \times  \prod_{i=i}^{k}\left(1-
  \tr \mathbf{B}_{i}\right)^{-n_{i}m/2-1}
  (dv)\bigwedge_{i=1}^{k}\left(d\mathbf{B}_{i}\right).
\end{equation}
This distribution shall be termed \emph{multimatrix variate Generalised Gamma-beta type I distribution}.
Moreover, the density function $dF_{\mathbf{B}_{1}, \dots, \mathbf{B}_{k}}(\mathbf{B}_{1},
\dots,\mathbf{B}_{k})$ can written as
$$
  \frac{\Gamma_{1}[Nm/2]}{\Gamma_{1}[n_{0}m/2]} \prod_{i=1}^{k}\left(\frac{|\mathbf{B}_{i}|^{(ni-m-1)/2}}{\Gamma_{m}[n_{i}/2]}\right )
  \left[ 1+\displaystyle\sum_{i=1}^{k} \frac{\tr\mathbf{B}_{i}}{\left(1-
  \tr\mathbf{B}_{i}\right)}\right ]^{-Nm/2}\hspace{3cm}
$$
\begin{equation}\label{b1}
  \hspace{5cm}
  \times  \prod_{i=i}^{k}\left(1-
  \tr \mathbf{B}_{i}\right)^{-n_{i}m/2-1}
  \bigwedge_{i=1}^{k}\left(d\mathbf{B}_{i}\right).
\end{equation}
and this marginal distribution shall be called \emph{multimatrix variate beta type I distribution}.
\end{theorem}
\begin{proof}
  A similar procedure to the proof of Theorem \ref{mggb2} can be applied here. Just consider the joint density functions (\ref{ggp2}) and (\ref{mp2}) under the change of variable  $\mathbf{B}_{i} = \mathbf{R}'_{i}\mathbf{R}_{i}$ with $i = 1, \dots,k$, and the derivation is complete.
\end{proof}

Finally

\begin{theorem}\label{ggW} 
Assume that $\mathbf{X} = \left(\mathbf{X}'_{0}, \dots, \mathbf{X}'_{k} \right)'$ has  a matrix variate spherical distribution, with $\mathbf{X}_{i} \in \Re^{n_{i} \times m}$, $n_{i} \geq m$, $i = 0,1, \dots, k$. Define $V = ||\mathbf{X}_{0}||^{2}$ and
$\mathbf{W}_{i} = \mathbf{X}'_{i}\mathbf{X}_{i} > \mathbf{0}$, $i = 1, \dots,k$. \newline Then, the joint density
$dF_{V,\mathbf{W}_{1}, \dots,\mathbf{W}_{k}}(v,\mathbf{X}_{1}, \dots,\mathbf{W}_{k})$ is given by
\begin{equation}\label{mggw}
   \frac{\pi^{Nm/2} v^{Nm/2-1}}{\Gamma_{1}[n_{0}m/2]}\prod_{i=1}^{k}\left(\frac{|\mathbf{W}_{i}|^{(ni-m-1)/2}}{\Gamma_{m}[n_{i}/2]}\right)
   h\left[v+\displaystyle\sum_{i=1}^{k}\tr \mathbf{W}_{i}\right]
   (dv)\bigwedge_{i=1}^{k}\left(d\mathbf{W}_{i}\right),
\end{equation}
where $V > 0$. This distribution shall be termed \emph{multimatrix variate generalised Gamma - generalised Wishart
distribution}. Similarly, let $\mathbf{W}_{i} = \mathbf{X}'_{i}\mathbf{X}_{i} > \mathbf{0}$, $i = 0, \dots,k$. Therefore, the joint density
$dF_{\mathbf{W}_{0}, \dots,\mathbf{W}_{k}}(\mathbf{X}_{0}, \dots,\mathbf{W}_{k})$ is 
\begin{equation}\label{mgw}
   \pi^{Nm/2}\left(\prod_{i=0}^{k}\frac{|\mathbf{W}_{i}|^{(ni-m-1)/2}}{\Gamma_{m}[n_{i}/2]}\right)
   h\left(\displaystyle\tr \sum_{i=1}^{k} \mathbf{W}_{i}\right)
   \bigwedge_{i=0}^{k}\left(d\mathbf{W}_{i}\right).
\end{equation}
\end{theorem}
\begin{proof}
From Theorem \ref{gge} we have
 \begin{equation*}
   \frac{\pi^{n_{0}m/2}}{\Gamma_{1}[n_{0}m/2]} h\left[v+\displaystyle\sum_{i=1}^{k}||\mathbf{X}_{i}||^{2}\right]
    v^{n_{0}m/2-1} (dv)\bigwedge_{i=1}^{k}\left(d\mathbf{X}_{i}\right).
\end{equation*} 
Defining $\mathbf{W}_{i} = \mathbf{X}'_{i}\mathbf{X}_{i}$ with $i = 1, \dots,k$ and proceeding as in the proof of Theorem \ref{mggb2}, the result (\ref{mggw}) is immediate. (\ref{mgw}) is obtained in \citet{dgcl:22}.
\end{proof}

Several multivariate density functions studied in the statistical literature are obtained as particular cases of the distributions proposed in this section. For example the distribution proposed in \citet[Problem 3.18, p.118]{mh:05}  and the distributions in \citet[Subsections 2.7.3.2 and 2.7.3.3, p. 54]{gv:93}.

\section{Some properties and generalisaton}\label{sec:4}

This section focus on multimatrix variate distributions for two and three matrix arguments. Additionally, if the matrix arguments are non-singular, the corresponding inverse distributions are also obtained. Despite a laborious task, the proposed novel method can consider several marginal distributions. The content of the section is an useful combination for multiple possible uses in real phenomena guide by experts with the corresponding previous knowledge of the marginal distributions.

\begin{theorem}\label{ggtpII} Assume that $\mathbf{X} = \left(\mathbf{X}'_{0},\mathbf{X}'_{1},\mathbf{X}'_{2}
\right)'$ has a matrix variate spherical distribution, with $\mathbf{X}_{i} \in \Re^{n_{i} \times m}$,
$n_{i} \geq m$, $i = 0,1, 2$. Define $V = ||\mathbf{X}_{0}||^{2}$, $\mathbf{T} = V
^{-1/2}\mathbf{X}_{1}$, and $\mathbf{R} = (V+||\mathbf{X}_{2}||)^{-1/2}\mathbf{X}_{2}$. The joint density
$dF_{V,\mathbf{T},\mathbf{R}}(v,\mathbf{T}, \mathbf{R})$ is given by
$$
  \frac{\pi^{n_{0}m/2}v^{Nm/2-1}}{\Gamma_{1}[n_{0}m/2]}
  h\left\{v \left[1+\left(1-||\mathbf{R}||^{2}\right)||\mathbf{T}||^{2}\right]\right\}\hspace{5cm}
$$
\begin{equation}\label{ggtp2}
  \hspace{5cm}\times \left(1-||\mathbf{R}||^{2}\right)^{(n_{0}+n_{1})m/2-1}(dv)\wedge d\mathbf{T})\wedge d\mathbf{R}).
\end{equation}
where $N = n_{0}+n_{1}+ n_{2}$, $V > 0$, $\mathbf{T}\in \Re^{n_{1} \times m}$, $\mathbf{R}\in \Re^{n_{2} \times m}$ and $||\mathbf{R}||^{2} \leq 1$. This distribution shall be termeded \emph{threematrix variate generalised  Gamma - Pearson type VII - Pearson type II distribution}. Moreover, the termed \emph{bimatrix variate Pearson type VII - Pearson type II distribution} is the marginal density $dF_{\mathbf{T},\mathbf{R}}(\mathbf{T}, \mathbf{R})$ of $\mathbf{T},\mathbf{R}$ and is given by
$$
  \frac{\Gamma_{1}[Nm/2]}{\pi^{(N-n_{0})m/2}\Gamma_{1}[n_{0}m/2]}
  \left[(1+\left(1-||\mathbf{R}||^{2}\right)||\mathbf{T}||^{2}\right]^{-Nm/2}\hspace{5cm}
$$
\begin{equation}\label{tp2}
  \hspace{5cm}\times \left(1-||\mathbf{R}||^{2}\right)^{(n_{0}+n_{1})m/2-1}(d\mathbf{T})\wedge d\mathbf{R}).
\end{equation}
\end{theorem}
\begin{proof}
By Theorem \ref{gge}
\begin{equation*}
   \frac{\pi^{n_{0}m/2}}{\Gamma_{1}[n_{0}m/2]} h\left[v+||\mathbf{X}_{1}||^{2}+||\mathbf{X}_{2}||^{2}\right]
    v^{n_{0}m/2-1} (dv)\wedge\left(d\mathbf{X}_{1}\right)\wedge\left(d\mathbf{X}_{2}\right),
\end{equation*}
where $V = ||\mathbf{X}_{0}||^{2} >0$. Let $V_{0}= V + ||\mathbf{X}_{2}||^{2}$, $\mathbf{T} = V^{-1/2}\mathbf{X}_{1}$, and $\mathbf{R} = (V+||\mathbf{X}_{2}||)^{-1/2}\mathbf{X}_{2}$. Hence $V= V_{0} - ||\mathbf{X}_{2}||^{2} = V_{0} - ||V_{0}\mathbf{R}_{2}||^{2} = V_{0}(1-||\mathbf{R}||^{2})$, $\mathbf{T} = [V_{0}(1-||\mathbf{R}||^{2})]^{-1/2}\mathbf{X}_{1}$, and $\mathbf{R} = V_{0}^{-1/2}\mathbf{X}_{2}$. Therefore, $(dv)=(dv_{0})$ and
$$
  (dv)\wedge\left(d\mathbf{X}_{1}\right)\wedge\left(d\mathbf{X}_{2}\right) = v_{0}^{(n_{1}+n_{2})m/2} (1-||\mathbf{R}||^{2})^{n_{1}m/2} (dv_{0})\wedge\left(d\mathbf{T}\right)\wedge\left(d\mathbf{R}\right)
$$
thus the desired result is obtained. The density (\ref{tp2}) is archived by integrating (\ref{ggtp2}) over $V>0$ via (\ref{int}).
\end{proof}

\begin{corollary}\label{ggfbII}
Under the Hypotheses of Theorem \ref{ggtpII}, define $\mathbf{F} = \mathbf{T}'\mathbf{T} >\mathbf{0}$ and $\mathbf{B}
= \mathbf{R}'\mathbf{R} >\mathbf{0}$. Then the joint density function $dF_{v,\mathbf{F},\mathbf{B}}(V,\mathbf{F},\mathbf{B})$ 
can be written as
$$
  \frac{\pi^{Nm/2}v^{Nm/2-1} |\mathbf{F}|^{(n_{1}-m-1)/2} |\mathbf{B}|^{(n_{2}-m-1)/2}}{\Gamma_{1}[n_{0}m/2]
  \Gamma_{m}[n_{1}/2] \Gamma_{m}[n_{2}m/2]}
  h\left\{v \left[1+\left(1-\tr\mathbf{B}\right)\tr\mathbf{F}\right]\right\}\hspace{5cm}
$$
\begin{equation}\label{ggfb2}
  \hspace{5cm}\times \left(1-\tr\mathbf{B}\right)^{(n_{0}+n_{1})m/2-1}(dv)\wedge d\mathbf{F})\wedge d\mathbf{B}).
\end{equation}
This distribution shall be termed \emph{threematrix variate generalised Gamma - beta type II - beta type
I distribution}. And the corresponding marginal density function
$dF_{\mathbf{F},\mathbf{B}}(\mathbf{F},\mathbf{B})$ is given by
$$
  \frac{\Gamma_{1}[Nm/2]|\mathbf{F}|^{(n_{1}-m-1)/2} |\mathbf{B}|^{(n_{2}-m-1)/2}}{\Gamma_{1}[n_{0}m/2]
  \Gamma_{m}[n_{1}/2] \Gamma_{m}[n_{2}m/2]}
  \left[1+\left(1-\tr\mathbf{B}\right)\tr\mathbf{F}\right]^{-Nm/2}\hspace{5cm}
$$
\begin{equation}\label{fb2}
  \hspace{5cm}\times \left(1-\tr\mathbf{B}\right)^{(n_{0}+n_{1})m/2-1}(d\mathbf{F})\wedge d\mathbf{B}).
\end{equation}
a distribution that shall be called \emph{bimatrix variate beta type II - beta type I distribution}.
\end{corollary}
\begin{proof}
Results (\ref{ggfb2}) and (\ref{fb2}) follow from  (\ref{ggtp2}) and (\ref{tp2}), respectively, by substitutions $\mathbf{F} = \mathbf{T}'\mathbf{T}$ and $\mathbf{B}= \mathbf{R}'\mathbf{R}$ and the use of \citet[Theorem 2.1.14, p. 66]{mh:05}. Thus,
$$
  (d\mathbf{T})\wedge d\mathbf{R}) = 2^{-2m}|\mathbf{F}|^{(n_{1}-m-1)/2} |\mathbf{B}|^{(n_{2}-m-1)/2} (d\mathbf{F})\wedge d\mathbf{B}) \bigwedge_{i=1}^{2}\left(\mathbf{H}'_{1_{i}}d\mathbf{H}_{1_{i}}\right).
$$
Then, the demonstration is complete after a repeated implementation of (\ref{evs}).
\end{proof}

Note that different families of distributions can be obtained by replicating the procedure of Theorem \ref{ggtpII}. 
For example, if the change of variable on $\mathbf{X}_{2}$ is not carried out the corresponding distribution obtained can be termed \textit{trimatrix variate generalised Gamma-elliptical-Pearson type II distribution}, and so on. Moreover, if $n_{2}= 1$ is set in Theorem \ref{ggtpII},  the corresponding distribution can be termed trimatrix variate generalised Gamma-Pearson type VII-vector Pearson type II distribution, a fact that shows the versatility of the method proposed in this paper. In this case, a simultaneous model can involve a random variable $V$, a random matrix $\mathbf{T}$ and a random vector $\mathbf{r} = V^{-1/2}\mathbf{x}_{2}$.

In addition, when the multimatrix variate distributions involve non singular matrix arguments, we can find their inverse distribution (or if their arguments are rectangular or singular matrices, we can propose their corresponding generalise inverse distributions). For the distribution defined in Corollary \ref{ggfbII}, there are different possible multimatrix variate inverse distributions according to the selected argument combination of 1, 2 or 3 matrices. The following case exhibits the inverse when the three arguments are considered.

\begin{theorem}\label{iggfbI} 
Assume the hypotheses of the Corollary \ref{iggfbI} and define $W = V^{-1}$, $\mathbf{A} =
\mathbf{F}^{-1}$ and $\mathbf{U} = \mathbf{B}^{-1}$. Then the joint density $dF_{W, \mathbf{A},
\mathbf{U}}(w,\mathbf{A},\mathbf{U})$ can be
written as
$$
  \frac{\pi^{Nm/2}v^{Nm/2-1} |\mathbf{A}|^{-n_{1}/2} |\mathbf{U}|^{-n_{2}/2}}{\Gamma_{1}[n_{0}m/2]
  \Gamma_{m}[n_{1}/2] \Gamma_{m}[n_{2}m/2]}
  h\left\{w^{-1} \left[1+\left(1-\tr\mathbf{U}^{-1}\right)\tr\mathbf{A}^{-1}\right]\right\}\hspace{5cm}
$$
\begin{equation}\label{iggfb2}
  \hspace{5cm}\times \left(1-\tr\mathbf{U}^{-1}\right)^{(n_{0}+n_{1})m/2-1}(dw)\wedge d\mathbf{A})\wedge d\mathbf{U}).
\end{equation}
This distribution shall be called \emph{threematrix variate inverse generalised Gamma - inverse beta type
II - inverse beta type I distribution}. And the marginal density function
$dF_{\mathbf{A},\mathbf{U}}(\mathbf{A},\mathbf{U})$ is given by
$$
  \frac{\Gamma_{1}[Nm/2]|\mathbf{A}|^{-n_{1}/2} |\mathbf{U}|^{-n_{2}/2}}{\Gamma_{1}[n_{0}m/2]
  \Gamma_{m}[n_{1}/2] \Gamma_{m}[n_{2}m/2]}
  \left[1+\left(1-\tr\mathbf{U}^{-1}\right)\tr\mathbf{A}^{-1}\right]^{-Nm/2}\hspace{5cm}
$$
\begin{equation}\label{ifb2}
  \hspace{5cm}\times \left(1-\tr\mathbf{U}^{-1}\right)^{(n_{0}+n_{1})m/2-1}(d\mathbf{A})\wedge d\mathbf{U}).
\end{equation}
A distribution that shall be termed \emph{bimatrix variate inverse beta type II - inverse beta type I
distribution}.
\end{theorem}
\begin{proof}
The density functions (\ref{iggfb2}) and (\ref{ifb2}) are achieved from (\ref{ggfb2}) and (\ref{fb2}) by defining $W = V^{-1}$, 
$\mathbf{A} = \mathbf{F}^{-1}$ and $\mathbf{U} = \mathbf{B}^{-1}$ and by applying \citet[Theorem 2.1.8, p. 59]{mh:05}. 
\end{proof}

Observe that the parameter domain of the multimatrix variate distributions can be extended to the complex or real fields. However their geometrical and/or statistical explication perhaps can be lost. These distributions are valid if we replace $n_{i}/2$ by  $a_{i}$, $n_{0}m/2$ by $a_{0}$ and $Nm/2$ by $a$. Where the $a'^{s}$ are complex numbers with positive real part. From a practical point of view for parametric estimation, this domain extension allows the use of nonlinear optimisation rather integer nonlinear optimisation, among other possibilities.

Moreover, note that each particular distribution is indexed by a kernel $h(\cdot)$ which defines itself another family of distributions. This opens a wide set of possible distributions available for an expert according to a preferred knowledge of a plausible $h(\cdot)$. 

Finally, each distribution can be reparametrized in order to reach no isotropy or general settings of its density function. In analogy with the normal case, the expressions obtained here appear in their standard form. However,  perform the substitution  $\mathbf{T}_{i}=\mathbf{A}_{i}^{-1}(\mathbf{S}_{i}-\boldsymbol{\mu}_{i})\mathbf{B}_{i}^{-1}/\sqrt{r_{i}}$  in (\ref{mp7}); where  $\mathbf{A}_{i} \in \Re^{n_{i} \times n_{i}}$, $\mathbf{B}_{i} \in \Re^{m \times m}$ are non-singular matrices  such that $\boldsymbol{\Sigma}_{i} = \mathbf{A}'_{i}\mathbf{A}_{i} > \mathbf{0}$ and $\boldsymbol{\Theta}_{i} = \mathbf{B}_{i}\mathbf{B}'_{i} > \mathbf{0}$, $\boldsymbol{\mu}_{i} \in \Re^{n_{i} \times m}$ and $r_{i} > 0$, $i = 1, \dots, k$. Then by \citet[Theorem 2.1.5, p. 58]{mh:05}, we obtain
$$
  (d\mathbf{T}_{i}) = r_{i}^{-n_{i}m/2}|\mathbf{\Sigma}_{i}|^{-m/2}|\mathbf{\Theta}_{i}|^{-n_{i}/2} (d\mathbf{S}_{i}).
$$
And finally, the density $dF_{\mathbf{S}_{1},\dots,\mathbf{S}_{k}}(\mathbf{S}_{1},\dots,\mathbf{S}_{k})$ can be written in the robust form
$$
   \frac{\Gamma_{1}[Nm/2]}{\pi^{(N-n_{0})m/2}\Gamma_{1}[n_{0}m/2] \left(\displaystyle\prod_{i=1}^{k} r_{i}^{n_{i}m/2}|\mathbf{\Sigma}_{i}|^{m/2}|\mathbf{\Theta}_{i}|^{n_{i}/2}\right)} \hspace{3cm}
$$
\begin{equation}\label{mmp2}
  \hspace{2cm}  
  \left[1+\tr\displaystyle\sum_{i=1}^{k}(\mathbf{S}_{i} - \boldsymbol{\mu}_{i})'\mathbf{\Sigma}_{i}^{-1}(\mathbf{S}_{i} - \boldsymbol{\mu}_{i}) \mathbf{\Theta}_{i}^{-1}\right]^{-Nm/2}
  \bigwedge_{i=1}^{k}\left(d\mathbf{S}_{i}\right).
\end{equation}
Observe that the particular values $k = 1$, and $Nm/2 = q$, ($N = n_{0}+n_{1}$), turns (\ref{mmp2}) into the matrix variate Pearson type VII distribution, see \citet[Subsection 2.7.3.2, p. 54]{gv:93}. And setting $q = (n_{1}m + r)/2$ in the preceding density, then the matrix variate T-distribution with $r$ degrees of freedom is achieved. Furthermore, taking $r = 1$ in the last density, then the matrix variate Cauchy distribution is obtained.

\section{Application of molecular docking in SARS-CoV-2}\label{sec:6}
Statistical applications supported by real data are usually difficult because the computation and underlying mathematical context force simplifications. Restrictions often related with inference based on independent sample information. It is extremely rare in nature to find time (space) probabilistic independent experiments. The most challenging problems in several areas are dependent in time and-or space. 
Some of them, such world pandemic situations are crucial and precise of urgent solutions. When a virus as SARS-CoV-2 arrives, a drug design is reacquired in a record time in order to avoid millions of deaths worldwide. A simplification of the docking problem involves two main steps: identify a target in the virus-protein, usually called pocket, and obtain a molecule (termed ligand) which must be place into an adequate site of the virus, for inhibiting its deathly activity. The target usually contains thousand of atoms with hundreds of possible unknown pockets were the small ligand can be placed. Discover the pocket and place the ligand is a colossal task for a human, however, the use of artificial intelligence can provide suitable sites to be considered by the biologist. Recently, a non conventional artificial intelligence method was provided in \citet{rbvc:22} by a suitable recent discipline termed shape theory. The docking problem requires the use of certain invariance as traslation, rotation and scaling, in order to avoid unnecesary computations. Instead of using the noisy Euclidean space latent in the classical artificial intelligence, the search is performed in quotient spaces with the demand symmetries. 
In this case, shape theory was applied in molecular docking for an automatic localization of ligand binding pockets in large proteins. 
The docking phenomena is modeled in \citet{rbvc:22} by a physics chemistry punctuation function based on Lennard-Jones potential type $6-12$ and $6-10$. It includes the expected time dependence where the matrix representation at each state of the pocket-ligand goes into more stable and optimal low energy. Then the probabilistic dependent sample records $n$ matrices in time ordered by a decreasing Lennard-Jones potential until the convergence is reached and the docking is stable. For this example, we have used the crystal structure of COVID-19 main protease an its inhibitor N3 due to \citet{jdx:20} and \citet{PDB:20}. The protein consists of 2387 atoms and the ligand is constituted by 21 atoms. Once \citet[Th. Sec. 3.]{rbvc:22} is applied, a plausible pocket is reached and the ligand is rotated rigidly into the cavity, until the energy is stabilized. In this case, $k=56$ locations of the ligand are dependently obtained inside a new cavity of 241 atoms. Figure \ref{fig1} shows the dependent sample for this example, they constitutes $21\times 3$ matrices $\mathbf{T}_{i}$, $i=1,\ldots,56$.

\begin{figure}
     \centering
        \animategraphics[
controls=play,
width=4cm,loop
]{1}{c}{1}{56}
         \caption{}
         \label{fig1}
\caption{Dependent sample of 56 movements of the rigid ligand N3 with 21 atoms inside a new cavity of 241 atoms found on SARS-CoV-2 main protease. For illustration, the ligand is oversized respect the cavity (gray); their 21 atoms are colored as red (C), green (N), blue (O). The animation shows the progressive optimisation of the molecular docking according to a decreasing Lennard-Jones potential. First a positive energy indicates that the coupling is repulsive, however it is diminishing until negative values of strong affinity are reached and stabilized into an attractive complex ligand-pocket.}
\end{figure}

Hence the  sample matrices $\mathbf{F}_{i} = \mathbf{T}'_{i}\mathbf{T}_{i}$ shall be the core for estimation of the parameters in the likelihood, corresponding to the joint distribution of $\mathbf{F}_{i}$ in equation (\ref{mb2}).

In this example, we consider a parameter space extension, from $n_{i}/2,i=0,\ldots,56$ to corresponding positive reals $a_{i},i=0,\ldots,56$. Then the optimisation problem is reduced to find the maximum likelihood estimators of $a_{0}$, and $a=a_{1}=\cdots=a_{56}$ in the following dependent sample joint distribution:
\begin{equation*}
     \frac{\Gamma_{1}[(a_{0}+ka)m]}{\Gamma_{1}[a_{0}m]\Gamma_{m}^{k}[a]}
  \prod_{i=1}^{k}|\mathbf{F}_{i}|^{a-\frac{m+1}{2}}
\left(1+\displaystyle\sum_{i=1}^{k}\tr\mathbf{F}_{i}\right)^{-(a_{0}+ka)m}
\end{equation*}
For computations we have used the free licensed software R. The optimisation routine is performed in the package Optimx. The following estimation is obtained consistently by several methods and different random seeds for the initial parameters:
$$a_{0}=0.34397, a=0.19735$$

This estimation is also important because the invariance of the likelihood, under the complete family of spherical distribution, avoids the difficult a priori knowledge of the underlying model. Also, no fitting test is required.

Note that the usual estimation based on an independent sample can not be applied here because the molecular docking is reached by the calibration of the Lennard-Jones potential which was set by the expert as a decreasing function of the dependent locus in the ligand.  

In the preceding example we have show that the joint distribution functions of Sections \ref{sec:3} and \ref{sec:4} are easily computable and applied to real data. This is an important opportunity for several similar phenomena with multimatrix dependent samples.

\section{Conclusions}

This work has defined the so termed \emph{multimatrix variate distributions} by setting in the general form of families of elliptical contours some isolated works on particular bimatrix laws. Most of the distributions are invariant under the class of spherical models, an important fact that avoids previous knowledge of the underlying multiple distribution. The distribution are also computable and several properties and combinations of new three matrices are also derived. The new theory solves the problem of matrix dependent sample by setting several ways of joint distributions with desirable marginal distributions under independence. The versatility of the method is such that both scalars, vectors and matrices can assemble in the same joint distribution simultaneously. The computability of the joint distribution also allows a further application in future by finding multiple probabilities on cones. Finally, the paper also provides applications in molecular docking of a known inhibitor into a new cavity of SARS-CoV-2 main protease.    



\end{document}